\documentclass[journal]{IEEEtran}
\IEEEoverridecommandlockouts                              % This command is only needed if 
\bstctlcite{IEEEexample:BSTcontrol}                                                          % you want to use the \thanks command
\usepackage{flushend} %%THIS command is for balancing the columns of last page
\usepackage{romannum}
\usepackage[usenames,dvipsnames]{xcolor}
\usepackage{tkz-berge}
\usetikzlibrary{fit,shapes}                                     
\usepackage{tikz}
\usepackage{calrsfs}
\DeclareMathAlphabet{\pazocal}{OMS}{zplm}{m}{n}
\usepackage{graphicx}
\usepackage{graphics} 
\usepackage{epsfig} 
\usepackage{mathptmx}
\usepackage{caption}
\usepackage{subcaption}
\usepackage{times} 
\usepackage{tikz}
\usetikzlibrary{positioning}
\usepackage{amsmath, amssymb}
\usepackage{amsthm}
\usepackage{amsfonts} 
\usepackage{setspace}
\usepackage{multirow}
\usepackage{textcomp}
\usepackage[english]{babel}
\usepackage{float} 
\usepackage{algorithmicx}
\usepackage[section]{algorithm}
\usepackage{algpascal}
\usepackage{algc}
\usepackage{algcompatible}
\usepackage{algpseudocode}
\let\oldReturn\Return
\renewcommand{\Return}{\State\oldReturn}
\usepackage{mathtools}
\usepackage{bm}
\usepackage{mathptmx}
\usepackage{times} 
\usepackage{mathtools, cuted}
\usepackage{textcomp}
\usepackage[english]{babel}
\usepackage{rotating}
\usepackage{pgfplots}
\pgfplotsset{compat=1.5}
\usepackage{siunitx}
\usepackage{pgfplotstable}
\usepackage{filecontents}

\newcommand{\A}{\pazocal{A}}
\newcommand{\B}{\pazocal{B}}
\newcommand{\C}{\pazocal{C}}
\newcommand{\K}{\pazocal{K}}
\newcommand{\W}{\pazocal{W}}
\newcommand{\Z}{\pazocal{Z}}
\newcommand{\D}{\pazocal{D}}
\newcommand{\E}{\pazocal{E}}

\newcommand{\U}{\pazocal{U}}

\newcommand{\J}{\pazocal{J}}
\newcommand{\I}{\pazocal{I}}

\renewcommand{\P}{\pazocal{P}}

\renewcommand{\S}{\pazocal{S}}
\newcommand{\x}{{X}}
\renewcommand{\u}{{U}}

\newcommand{\remove}[1]{}
\def \cN{{\mathcal N}^t}
\def \cNn{{\mathcal N}^b}

\def \*{\star}
\def \10n{\!\!\!\!\!\!\!\!\!\!}

\newcommand{\bK}{\bar{K}}

\newcommand{\R}{\mathbb{R}}
\newcommand{\bA}{\bar{A}}
\newcommand{\bB}{\bar{B}}
\newcommand{\bC}{\bar{C}}

\newtheorem{thm}{Theorem}
\newtheorem{lem}{Lemma}

\newtheorem{defn}{Definition}
\newtheorem{cor}{Corollary}

\newtheorem{prob}{Problem}

\newtheorem{prop}{Proposition}
\graphicspath{{Figures/}}

\title{\LARGE \bf Approximating Constrained Minimum Cost Input-Output Selection for Generic Arbitrary Pole Placement in Structured Systems}
\author{Shana~Moothedath,
        Prasanna~Chaporkar
        and~Madhu~N.~Belur% <-this % stops a space
\thanks{The authors are in the Department of Electrical Engineering, Indian Institute of Technology Bombay, India. Email: $\lbrace$shana, chaporkar, belur$\rbrace$@ee.iitb.ac.in. This work was supported in part by SERB (DST) and BRNS, India.}}
\begin{document}
\maketitle
\thispagestyle{empty}
\pagestyle{empty}
%%%%%%%%%%%%%%%%%%%%%%%%%%%%%%%%%%%%%%%%%%%%%%%%%%%%%%%%%%%%%%%%%%%%%%%%%%%%%%555
\begin{abstract}
This paper is about minimum cost constrained selection of inputs and outputs in structured systems for generic arbitrary pole placement. The input-output set is constrained in the sense that the set of states that each input can influence and the set of states that each output can sense is pre-specified. Our goal is to optimally select an input-output set that the system has no structurally fixed modes. Polynomial time algorithms do not exist for solving this problem unless P~=~NP. To this end, we propose an approximation algorithm by splitting the problem in to three sub-problems: a)~minimum cost accessibility problem, b)~minimum cost sensability problem and c)~minimum cost disjoint cycle problem. We prove that problems~a)~and~b) are equivalent to the weighted set cover problem. We also show that problem~c) can be solved using a minimum cost perfect matching algorithm. Using these, we give an approximation algorithm which solves the minimum cost generic arbitrary pole placement problem. The proposed algorithm incorporates an approximation algorithm to solve the weighted set cover problem to solve a)~and~b) and a minimum cost perfect matching algorithm to solve c). Further, we show that the algorithm has polynomial complexity  and gives an order optimal $O({\rm log}n)$ approximate solution to the minimum cost input-output selection for generic arbitrary pole placement problem, where $n$ denotes the number of states in the system.
\end{abstract}
%%%%%%%%%%%%%%%%%%%%%%%%%%%%%%%%%%%%%%%%%%%%%%%%%%%%%%%%%%%%%
%%%%%%%%%%%%%%%%%%%%%%%%%%%%%%%%%%%%%%%%%%%%%%%%%%%%%%%%%%%5
\section{Introduction}\label{sec:intro}
Consider structured matrices $\bA \in\{\*,0\}^{n \times n}$, $\bB \in\{\*,0\}^{n \times m}$ and $\bC \in\{\*,0\}^{p \times n}$ whose entries are either $\*$ or $0$. The matrices $\bA$, $\bB$ and $\bC$ {\it structurally represent} state, input and output matrices respectively of any control system $\dot{x} = Ax + Bu$, $y = Cx$ such that:
\begin{eqnarray}\label{eq:struc}
A_{ij} &=& 0 \mbox{~whenever~} \bA_{ij} = 0,\mbox{~and} \nonumber \\
B_{ij} &=& 0 \mbox{~whenever~} \bB_{ij} = 0,\mbox{~and} \nonumber \\
C_{ij} &=& 0 \mbox{~whenever~} \bC_{ij} = 0.
\end{eqnarray}
Any triple $(A,B,C)$ that satisfy \eqref{eq:struc} is said to be a {\it numerical realization} of the {\it structural system} $(\bA$, $\bB$, $\bC)$. Further, the matrix $\bK \in \{\*,0\}^{m \times p}$, where $\bK_{ij} = \*$ if the $j^{\rm th}$ output is available for static output feedback to the $i^{\rm th}$ input is referred as the {\it feedback matrix}. Let $[K]$ is the collection of all numerical realizations of $\bK$, i.e., $[K] := \{K:K_{ij} = 0 \mbox{~if~} \bK_{ij} = 0\}$. 
\begin{defn}\label{def:sfm}
The structural system $(\bA, \bB, \bC)$ is said not to have structurally fixed modes (SFMs) with respect to an information pattern $\bK$ if there exists one numerical realization $(A, B, C)$ of $(\bA, \bB, \bC)$ such that $\cap_{K \in [K]} \sigma(A + BKC) = \phi$, where the function $\sigma(T)$ denotes the set of eigenvalues of any square matrix $T$.
\end{defn}
 Let $p_u \in \R^m$, where every entry $p_u(i)$, $i=1,\ldots,m$, indicates the cost of using $i^{\rm th}$ input. Also, $p_y \in \R^p$, where every entry $p_y(j)$, $j=1,\ldots,p$, indicates the cost of using $j^{\rm th}$ output. For $\W \subseteq \{1,\ldots,m\}$, $\Z \subseteq \{1,\ldots,p\}$, let $\bB_\W$ be the restriction of $\bB$ to columns only in $\W$ and $\bC_\Z$ be the restriction of $\bC$ to rows only in $\Z$. Furthermore, let $\K = \{(\W, \Z) : (\bA,\bB_\W, \bC_\Z, \bK_{(\W \times \Z)}) \mbox{ has no structurally fixed modes}\}$. Our aim is to find $(\I, \J) \in \K$ such that the cost of inputs and outputs is minimized. Specifically, we wish to solve the following optimization: for any $(\I, \J)$, define $p(\I, \J) = \sum_{i\in\I}p_u(i) + \sum_{j\in\J}p_y(j)$. 
\begin{prob}\label{prob:one}
Given a structural system $(\bA, \bB, \bC)$, feedback matrix $\bK$ and cost vectors $p_u$, $p_y$, find
\[ (\I^\*, \J^\*) ~\in~ \arg\min_{\10n (\I, \J) \in \K} p(\I, \J). \] 
\end{prob}

We refer to Problem~\ref{prob:one} as {\it minimum cost constrained input-output selection for generic arbitrary pole placement}. Let $p^\* = p(\I^\*, \J^\*)$. Thus, $p^\*$ denotes the minimum cost for constrained input-output selection that ensures generic arbitrary pole placement. Without loss of generality, assume $(\bA, \bB, \bC, \bK)$ has no SFMs. Thus $\K$ is non-empty.

 In this paper we consider a special case in which $\bK$ is complete, i.e., $\bK_{ij} = \*$ for all $i,j$. Even with this restriction the problem is NP-hard. In our main contribution, we propose an approximation algorithm of computational complexity $O(n^3)$. In the worst case, the proposed algorithm achieves approximation ratio of $6~{\rm log\,}n$, and the ratio can be improved significantly in many practical systems. We also establish a negative result which states that no polynomial time algorithm can achieve approximation ratio of $\frac{1}{4}{\rm log\,}n$. Thus our algorithm is order optimal as it provides $O({\rm log\,}n)$ approximation. Formally, the main result of our paper is the following:
    
\begin{thm}\label{th:main}
Consider a structural system $(\bA, \bB, \bC)$, a complete feedback matrix $\bK$ and cost vectors $p_u, p_y$. Let $n$ be the number of states in the system and $(\I_a, \J_a)$ be an output of Algorithm~\ref{alg:threestage}. Then the following hold:
\begin{itemize}
\item[i)] $(\I_a, \J_a) \in \K$, i.e., $(\bA, \bB_{\I_a},\bC_{\J_a}, \bK_{(\I_a \times \J_a)})$ has no SFMs, and
\item[ii)] $p(\I_a, \J_a) \leqslant (2\,{\rm log\,}n) \,p^\*$,
\end{itemize}
Moreover, there does not exist any polynomial time algorithm to solve Problem~\ref{prob:one} that has approximation ratio $(1-o(1))\,{\rm log\,}n$. Thus the proposed algorithm is an order optimal approximation algorithm.
\end{thm}

The organization of this paper is as follows: in Section~\ref{sec:graph} we discuss preliminaries, existing results and related work in this area. In Section~\ref{sec:minCC} we explain our approach to solve the minimum cost  input-output selection problem for generic arbitrary pole placement by splitting it in to three sub-problems: minimum cost accessibility, minimum cost sensability and minimum cost disjoint cycle problem. In Section~\ref{sec:main} we discuss an approximation algorithm for solving the problem and then prove the main results of the paper. In Section~\ref{sec:cases} we explain the approximation result in the context of few special cases. In Section~\ref{sec:conclu} we give the final concluding remarks.
%%%%%%%%%%%%%%%%%%%%%%%%%%%%%%%%%%%%%%%%%%%%%%%%%%%%%%%%%%%%%%%%%%%%%5
\section{ Preliminaries, Existing Results and Related Work}\label{sec:graph}
In this section we first discuss few graph theoretic concepts used in the sequel and some existing results. Then we discuss related work in this area.

\subsection{Preliminaries and Existing Results} 
Arbitrary pole placement is said to be possible in a structural system if it has no structurally fixed modes (SFMs). Basically there are two types of fixed modes, \mbox{Type-1} and \mbox{Type-2} (see \cite{WanDav:73}, \cite{PicSezSil:84} for more details). To ensure non-existence of SFMs one has to ensure that both these types are absent in the system. Presence of Type-1 SFMs can be checked using the concept of strong connectedness of the system digraph which is constructed as follows: firstly, we construct the state digraph $\D(\bA) := \D(V_{\x}, E_{\x})$, where $V_{\x} = \{x_1, \ldots, x_n \}$ and $(x_j, x_i) \in E_{\x}$ if $\bA_{ij} \neq 0$. Thus a directed edge $(x_j, x_i)$ exists if state $x_j$ can influence state $x_i$. Now we construct the system digraph $\D(\bA, \bB, \bC, \bK) := \D(V_{\x}\cup V_{\u}\cup V_{Y}, E_{\x}\cup E_{\u}\cup E_{Y}\cup E_{K})$, where $V_{U} = \{u_1, \ldots, u_m \}$ and $V_{Y} = \{y_1, \ldots, y_p \}$. An edge $(u_j, x_i) \in E_{U}$ if $\bB_{ij} \neq 0$, $(x_j, y_i) \in E_{Y}$ if $\bC_{ij} \neq 0$ and $(y_j, u_i) \in E_K$ if $\bK_{ij} \neq 0$. Thus a directed edge $(u_j, x_i)$ exists if input $u_j$ can actuate state $x_i$ and a directed edge $(x_j, y_i)$ exists if output $y_i$ can sense state $x_j$. Construction of state digraph $\D(\bA)$ and system digraph $\D(\bA, \bB, \bC, \bK)$ is illustrated through an example in Figure \ref{fig:eg}. Next we define two concepts, namely accessibility and sensability, that we need for explaining our algorithm.
\begin{defn}\label{def:acc}
A state $x_i$ is said to be {\it accessible} if there exists a directed simple path from some input $u_j$ to $x_i$ in the digraph $\D(\bA, \bB, \bC, \bK)$. Also, a state $x_i$ is said to be {\it sensable} if there exists a directed simple path from $x_i$ to some output $y_j$ in the digraph $\D(\bA, \bB, \bC, \bK)$.
\end{defn}

A digraph is said to be strongly connected if for each ordered pair of vertices $(v_1,v_k)$
there exists an elementary path from $v_1$ to $v_k$. A strongly connected component (SCC) of a digraph is a maximal strongly connected subgraph of it. If $\D(\bA)$ is a single SCC, then the system is said to be {\it irreducible}.
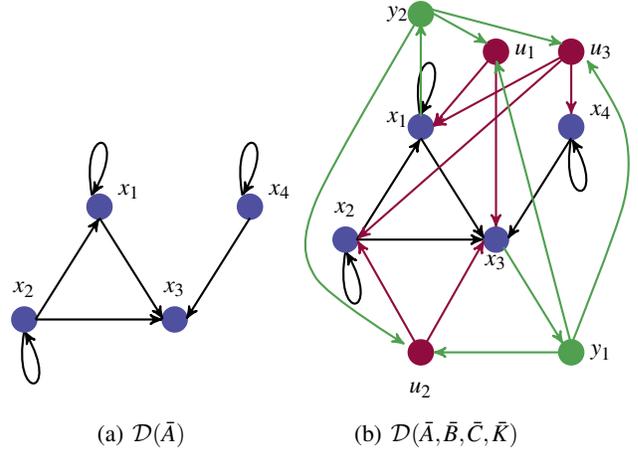
\begin{figure}[t]
%\centering
\begin{equation*} \label{eq:Amatrix}
\begin{array}{ll}
\bA = 
\begin{bmatrix}
\* & \* & 0 & 0\\
0 & \* & 0 & 0\\
\* & \* & 0 & \* \\
0 & 0 & 0 & \* \\[-4mm]
\hphantom{56} & \hphantom{56}& \hphantom{56}& \hphantom{56}
\end{bmatrix}, & 
\bB = 
\begin{bmatrix}
\* & 0 & \* \\
0 & \* & \* \\
\* & \* & 0  \\
0 & 0 & \* \\[-4mm]
\hphantom{56} & \hphantom{56}& \hphantom{56}
\end{bmatrix}, \\
\bC = 
\begin{bmatrix}
0 & 0 & \* & 0 \\
\* & 0 & 0 & 0 \\[-4mm]
\hphantom{56} & \hphantom{56}& \hphantom{56}& \hphantom{56}
\end{bmatrix}, & 
\bK = 
\begin{bmatrix}
\* & \* \\
\* &\* \\
\* & \* \\[-4mm]
\hphantom{56} & \hphantom{56}
\end{bmatrix}.
\end{array}
\end{equation*}
\begin{subfigure}[b]{0.2\textwidth}
\begin{tikzpicture}[->,>=stealth',shorten >=0.3pt,auto,node distance=1.5cm,
                thick,main node/.style={circle,draw,font=\Large\bfseries}]

\tikzset{every loop/.style={min distance=10mm,looseness=10}}
    
\path[->] (-0.1,0.15) edge [in=60,out=100,loop] node[auto] {} ();
\path[->] (1.9,0.15) edge [in=60,out=100,loop] node[auto] {} ();
\path[->] (-1.0,-1.65) edge [in=-60,out=-100,loop] node[auto] {} ();
                
\definecolor{myblue}{RGB}{80,80,160}
\definecolor{mygreen}{RGB}{80,160,80}
\definecolor{myred}{RGB}{144, 12, 63}
\definecolor{myyellow}{RGB}{214, 137, 16}

  \node at (0.4,0.2) {\small $x_1$};
  \node at (-1,-1.1) {\small $x_2$};
  \node at (1,-1.1) {\small $x_3$};
  \node at (2.4,0.2) {\small $x_4$};
    
  \fill[myblue] (0,0) circle (5.0 pt);
  \fill[myblue] (2.0,0) circle (5.0 pt);
  \fill[myblue] (1.0,-1.5) circle (5.0 pt);
  \fill[myblue] (-1.0,-1.5) circle (5.0 pt);
  
  \draw (0,-0.15)  ->  (0.85,-1.5);
  \draw (-0.85,-1.5)  ->  (0.85,-1.5);
  \draw (-0.85,-1.5)  ->  (0,-0.15);
  \draw (2,-0.15)  ->  (1.15,-1.5);
     \end{tikzpicture}
\caption{$\D(\bA)$}
\label{fig:digraph}
\end{subfigure}~\hspace{0.5 mm}
\begin{subfigure}[b]{0.2\textwidth}
\begin{tikzpicture}[->,>=stealth',shorten >=0.5pt,auto,node distance=2cm,
                thick,main node/.style={circle,draw,font=\Large\bfseries}]

\tikzset{every loop/.style={min distance=10mm,looseness=10}}
    
\path[->] (-0.01,0.15) edge [in=60,out=100,loop] node[auto] {} ();
\path[->] (2,-0.15) edge [in=-60,out=-100,loop] node[auto] {} ();
\path[->] (-1.0,-1.65) edge [in=-60,out=-100,loop] node[auto] {} ();
                
\definecolor{myblue}{RGB}{80,80,160}
\definecolor{mygreen}{RGB}{80,160,80}
\definecolor{myred}{RGB}{144, 12, 63}
\definecolor{myyellow}{RGB}{214, 137, 16}

  \node at (-0.3,0.1) {\small $x_1$};
  \node at (-1,-1.1) {\small $x_2$};
  \node at (1,-1.8) {\small $x_3$};
  \node at (2.4,0.2) {\small $x_4$};
    
  \fill[myblue] (0,0) circle (5.0 pt);
  \fill[myblue] (2.0,0) circle (5.0 pt);
  \fill[myblue] (1.0,-1.5) circle (5.0 pt);
  \fill[myblue] (-1.0,-1.5) circle (5.0 pt);
  
  \draw (0,-0.15)  ->  (0.85,-1.5);
  \draw (-0.85,-1.5)  ->  (0.85,-1.5);
  \draw (-0.85,-1.5)  ->  (0,-0.15);
  \draw (2,-0.15)  ->  (1.15,-1.5);

  \node at (1.4,1) {\small $u_1$};
  \node at (0,-3.5) {\small $u_2$};
  \node at (2.4,1) {\small $u_3$};

  \fill[myred] (1,1) circle (5.0 pt);
  \draw [myred] (1,1)  ->   (0.15,0);
  \draw [myred] (1,1)  ->   (1,-1.35);
  \fill[myred] (0,-3) circle (5.0 pt);
  \draw [myred] (0,-3)  ->   (0.85,-1.5);
  \draw [myred] (0,-3)  ->   (-0.85,-1.5);  
  \fill[myred] (2,1) circle (5.0 pt);
  \draw [myred] (2,1)  ->   (2,0.15);
  \draw [myred] (2,1)  ->   (0.15,0);
  \draw [myred] (2,1)  ->   (-0.85,-1.5); 
  
  \node at (2.4,-3) {\small $y_1$};
  \node at (-0.3,1.5) {\small $y_2$};
    
  \fill[mygreen] (2,-3) circle (5.0 pt); 
  \fill[mygreen] (0,1.5) circle (5.0 pt);
  
  \draw [mygreen] (1.1,-1.6)  ->   (1.9,-2.9);
  \draw [mygreen] (1.9,-3)  ->   (0.15,-3);
  \draw [mygreen] (0,0.14)  ->   (0,1.4);
  \draw [mygreen] (0,1.6)  ->   (0.9,1.1); 
  \draw [mygreen] (0,1.6)  ->   (1.9,1.1);
   
\draw [mygreen] (2,-2.9)  ->   (1.0,0.9);
\draw [mygreen] plot [smooth,->] coordinates {(2.0,-2.9) (2.7,-1) (2.7,0) (2.2,0.90) };
\draw [mygreen] plot [smooth] coordinates {(0.0,1.5) (-1.4,-1) (-1.4,-2) (-0.2,-2.9)  };
\end{tikzpicture}
\caption{$\D(\bA, \bB, \bC, \bK)$}
\label{fig:systemdigraph}
\end{subfigure}
\caption{The state digraph and system digraph representation of the structured system $(\bA, \bB, \bC, \bK)$ is shown in Figure~\ref{fig:digraph} and Figure~\ref{fig:systemdigraph} respectively.}
\label{fig:eg}
%\end{center}
%\end{figure}
\end{figure}  
Using the digraph $\D(\bA, \bB, \bC, \bK)$ a necessary and sufficient graph theoretic condition for absence of SFMs is given in the following result.

\begin{prop} [\cite{PicSezSil:84}, Theorem 4]\label{prop:SFM1} 
A structural system $(\bA, \bB, \bC)$ has no structurally fixed modes with respect to a feedback matrix $\bK$ if and only if the following conditions hold:\\
\noindent a)~in the digraph $\D(\bA, \bB, \bC, \bK)$, each state node $x_i$ is contained in an SCC which includes an edge in $E_K$, and \\
\noindent b)~there exists a finite disjoint union of cycles $\C_g = (V_g, E_g)$ in $\D(\bA, \bB, \bC, \bK)$ where $g$ 
is a positive integer such that $V_X \subset \cup_{g}V_g$.
\end{prop}

In Proposition~\ref{prop:SFM1}, condition~a) corresponds to SFMs of Type~1 and condition~b) corresponds to SFMs of Type~2. In order to characterize condition~a) we first generate a {\it directed acyclic graph} (DAG) associated with $\D(\bA)$ by condensing each SCC to a supernode. In this DAG, vertex set comprises of all SCCs in $\D(\bA)$. A directed edge exists between two nodes of the DAG {\it if and only if} there exists a directed edge connecting two states in the respective SCCs in $\D(\bA)$. Using this DAG we have the following definition that characterizes SCCs in $\D(\bA)$.

\begin{defn}\label{def:scc}
An SCC is said to be \underline{linked} if it has atleast one incoming or outgoing edge from another SCC. Further, an SCC is said to be \underline{non-top} \underline{linked} (\underline{non-bottom} \underline{linked}, resp.) if it has no incoming (outgoing, resp.) edges to (from, resp.) its vertices from (to, resp.) the vertices of another SCC.
\end{defn}

Without loss of generality we will assume that $\D(\bA)$ has $q$ non-top linked SCCs, $\cN_1, \ldots,\cN_q$ and $k$ non-bottom linked SCCs, $\cNn_1, \ldots, \cNn_k$. We have the following definition.

\begin{defn}\label{def:cover}
An SCC is said to be covered by input $u_j$ if there exists a state $x_i$ in the SCC such that $\bB_{ij} = \*$. Similarly, an SCC is said to be covered by output $y_j$ if there exists a state $x_i$ in the SCC such that $\bC_{ji} = \*$.
 \end{defn}
 We define $\mu_i := \{j:\cN_j \mbox{~is covered by~} u_i\}$ and $\eta_i := \{j:\cNn_j \mbox{~is covered by~} y_i\}$. Let $\mu_{\rm max} := {\rm max}_{ i} \mu_i$ and $\eta_{\rm max} := {\rm max}_i \eta_i$. In the example given in Figure~\ref{fig:digraph} each state is individually an SCC. Moreover, there are two non-top linked SCCs, $\cN_1 = x_2$ and $\cN_2 = x_4$ and one non-bottom linked SCC, $\cNn_1 = x_3$. Note that $x_1$ is neither non-top linked nor non-bottom linked SCC. Also, $\mu_1 = 0$, $\mu_2 = 1$, $\mu_3 = 2$, $\eta_1 = 1$ and $\eta_2 =0$. Thus $\mu_{\rm max} = 2$ and $\eta_{\rm max}=1$.
Following is an important observation.
\begin{cor}\label{cor:acc}
All states are accessible (sensable, resp.) if all non-top (non-bottom, resp.) linked SCCs are covered by input (output, resp.).
\end{cor}
Corollary~\ref{cor:acc} is an immediate consequence of Definitions~\ref{def:acc}~and~\ref{def:scc}. For a generic system $(\bA, \bB, \bC)$ with feedback matrix $\bK$, verifying absence of SFMs has polynomial complexity. Specifically, condition~a) can be verified in $O(n^2)$ computations using the concept of SCCs in $\D(\bA, \bB, \bC, \bK)$ \cite{AhoHop:74}. Condition~b) can be verified in $O(n^{2.5})$ computations using concepts of information paths given in \cite{PapTsi:84} or using bipartite matching as proposed in \cite{PeqKarPap:15}. In our work, we use bipartite matching condition and so we explain this in detail.

 Given an undirected bipartite graph $G(V, \widetilde{V}, \E)$, where $V \cup \widetilde{V}$ denotes the set of vertices and $\E \subseteq V \times \widetilde{V}$ denotes the set of edges, a matching $M$ is a collection of edges $M \subseteq \E$ such that for any two edges $(i,j), (u,v) \in M$, $i \neq u$ and $j \neq v$. A perfect matching is a matching $M$ such that $|M| = {\rm min}(|V|, |\widetilde{V}|)$. Now for checking condition~b) in a structural system, we use the bipartite graph $\B(\bA, \bB, \bC, \bK)$ constructed in \cite{PeqKarPap:15}. Let $\B(\bA, \bB, \bC, \bK) := \B(V_{X'} \cup V_{U'} \cup V_{Y'}, V_{X}\cup V_{U} \cup V_{Y}, \E_{\x} 
 \cup \E_{\u} \cup \E_Y \cup \E_K \cup \E_\mathbb{U} \cup \E_\mathbb{Y})$, where $V_{X'}=\{x'_1, \ldots, x'_n \}$, $V_{U'}=\{u'_1, \ldots, u'_m \}$, $V_{Y'} = \{y'_1, \ldots, y'_p \}$ and $V_{X}=\{x_1, \ldots, x_n \}$, $V_{U}=\{u_1, \ldots, u_m \}$ and $V_{Y} = \{y_1, \ldots, y_p \}$. Also, $(x'_i, x_{j}) \in \E_{\x} \Leftrightarrow (x_j, x_i) \in E_{\x}$, $(x'_i, u_{j}) \in \E_{\u} \Leftrightarrow (u_j, x_i) \in E_{\u}$, $(y'_{j}, x_i) \in \E_{Y} \Leftrightarrow (x_i, y_j) \in E_{Y}$ and $(u'_i, y_{j}) \in \E_{K} \Leftrightarrow (y_j, u_i) \in E_{K}$. Moreover, $\E_\mathbb{U}$ include edges $(u'_i, u_i)$ for $i =1,\ldots,m$ and  $\E_\mathbb{Y}$ include edges $(y'_j, y_j)$ for $j =1,\ldots,p$. We show that there exists a perfect matching in $\B(\bA, \bB, \bC, \bK)$ if and only if the system $(\bA, \bB, \bC)$ along with feedback matrix $\bK$ satisfies condition~b) (see Section~\ref{sec:main}).

\begin{figure}[t]
\begin{subfigure}[b]{0.4\textwidth}
\centering
\definecolor{myblue}{RGB}{80,80,160}
\definecolor{mygreen}{RGB}{80,160,80}
\definecolor{myred}{RGB}{144, 12, 63}
\definecolor{myyellow}{RGB}{214, 137, 16}
\definecolor{aquamarine}{rgb}{0.5, 1.0, 0.83}
%\begin{figure} 
%\begin{center}
\begin{tikzpicture} [scale = 0.2]               
          \node at (10,-4.5) {\small $x'_1$};
          \node at (10,-7.0) {\small $x'_2$};
          \node at (10,-9.5) {\small $x'_3$};
          \node at (10,-12.0) {\small $x'_4$};
          \node at (10,-15) {\small $u'_1$};
          \node at (10,-17.5) {\small $u'_2$};
          \node at (10,-20) {\small $u'_3$};
          \node at (10,-22.5) {\small $y'_1$};
          \node at (10,-25) {\small $y'_2$};           
         
          \fill[myblue] (12,-5) circle (30.0 pt); 
          \fill[myblue] (12,-7.5) circle (30.0 pt);
          \fill[myblue] (12,-10) circle (30.0 pt);
          \fill[myblue] (12,-12.5) circle (30.0 pt);
          \fill[myred] (12,-15) circle (30.0 pt);
          \fill[myred] (12,-17.5) circle (30.0 pt);
          \fill[myred] (12,-20) circle (30.0 pt); 
          \fill[mygreen] (12,-22.5) circle (30.0 pt);
          \fill[mygreen] (12,-25) circle (30.0 pt); 
                    
\draw (13,-5)  --   (31,-5);
\draw (13,-5)  --   (31,-7.5);
\draw (13,-7.5)  --   (31,-7.5);
\draw (13,-10)  --   (31,-5);
\draw (13,-10)  --   (31,-7.5);
\draw (13,-10)  --   (31,-12.5);
\draw (13,-12.5)  --   (31,-12.5);

\draw (13,-5)  --   (31,-15);
\draw (13,-10)  --   (31,-15);
\draw (13,-7.5)  --   (31,-17.5);
\draw (13,-10)  --   (31,-17.5);
\draw (13,-5)  --   (31,-20);
\draw (13,-7.5)  --   (31,-20);
\draw (13,-12.5)  --   (31,-20);

\draw[myred] (13,-15)  --   (31,-15);
\draw[myred] (13,-17.5)  --   (31,-17.5);
\draw[myred] (13,-20)  --   (31,-20);

\draw[mygreen] (13,-22.5)  --   (31,-22.5);
\draw[mygreen] (13,-25)  --   (31,-25);

\draw[myred] (13,-15)  --   (31,-22.5);
\draw[myred] (13,-15)  --   (31,-25);
\draw[myred] (13,-17.5)  --   (31,-22.5);
\draw[myred] (13,-17.5)  --   (31,-25);
\draw[myred] (13,-20)  --   (31,-22.5);
\draw[myred] (13,-20)  --   (31,-25);

\draw[mygreen] (13,-22.5)  --   (31,-10);
\draw[mygreen] (13,-25)  --   (31,-5);

          \node at (34,-5) {\small $x_1$};
          \node at (34,-7.5) {\small $x_2$};
          \node at (34,-10) {\small $x_3$};
          \node at (34,-12.5) {\small $x_4$};
          \node at (34,-15) {\small $u_1$};
          \node at (34,-17.5) {\small $u_2$};
          \node at (34,-20) {\small $u_3$};
          \node at (34,-22.5) {\small $y_1$};
          \node at (34,-25) {\small $y_2$};
          
          \fill[myblue] (32,-5) circle (30.0 pt); 
          \fill[myblue] (32,-7.5) circle (30.0 pt);
          \fill[myblue] (32,-10) circle (30.0 pt);
          \fill[myblue] (32,-12.5) circle (30.0 pt);
          \fill[myred] (32,-15) circle (30.0 pt);
          \fill[myred] (32,-17.5) circle (30.0 pt);                                
          \fill[myred] (32,-20) circle (30.0 pt);
          \fill[mygreen] (32,-22.5) circle (30.0 pt);
          \fill[mygreen] (32,-25) circle (30.0 pt);                                

        \node at (12,-3) {$V_{X'} \cup V_{U'} \cup V_{Y'}$};
        \node at (32,-3) {$V_{X} \cup V_{U} \cup V_{Y}$};
\end{tikzpicture}
\end{subfigure}
\caption{The bipartite digraph representation $\B(\bA, \bB, \bC, \bK)$ of the structured system $(\bA, \bB, \bC, \bK)$ shown in Figure \ref{fig:eg}.}
\label{fig:system_bigraph}
%\end{center}
%\end{figure}
\end{figure}
 Note that $\D(\bA)$, $\D(\bA,\bB, \bC, \bK)$ are digraphs, but $\B(\bA, \bB, \bC, \bK)$ is an undirected graph. Also, $E$ denotes directed edges and $\E$ denotes undirected edges. The system bipartite graph $\B(\bA, \bB, \bC, \bK)$ for the structural system given in Figure~\ref{fig:eg} is shown in Figure~\ref{fig:system_bigraph}. Summarizing, a structural closed-loop system is said not to have SFMs if and only if all state vertices lie in some SCC of $\D(\bA, \bB, \bC, \bK)$ with an edge in $E_K$ and the system bipartite graph $\B(\bA, \bB, \bC, \bK)$ has a perfect matching. Thus, using the two graph theoretic conditions explained in this section, we conclude that presence of SFMs in a structural closed-loop system can be checked in $O(n^{2.5})$ computations. Hence one can conclude if generic arbitrary pole placement is possible in a structural system in polynomial time. However, optimal selection of input-output set that guarantee arbitrary pole placement cannot be solved in polynomial time unless P~=~NP \cite{PeqKarPap:15}.
 
 \subsection{Related Work}
 In large scale systems, including biological systems, the web, power grids and social network to name a few, more often only the connections in the graph are known. The exact parameters are unavailable. In this context, structural analysis of the system is performed to study the various system properties generically (see \cite{Lin:74}, \cite{Mur:87}, \cite{PapTsi:84}, \cite{PicSezSil:84} and references therein). Study of controllability and observability of the system generically using the structure of the system is referred to as {\it structural controllability} and {\it structural observability}. Structural controllability was introduced by Lin in \cite{Lin:74}. Since then various associated problems including minimum input selection \cite{LiuBar:16}, \cite{ComDio:15}, \cite{Ols:15} and \cite{PeqKarAgu_2:16}, input addition for structural controllability \cite{ComDio:13}, strong structural controllability \cite{ChaMes:13}, minimum cost control selection and control configuration selection \cite{PeqKarPap:15} are addressed in the literature. In most of these papers the structure of the input (output, resp.) matrix is not constrained. For example \cite{PeqKarAgu_2:16} discusses the problem of finding sparsest set $(\bB,\bC,\bK)$ for a given $\bA$ such that arbitrary pole placement is possible. This problem can be solved in polynomial complexity. However, constrained input (output, resp.) selection for structural controllability (observability, resp.) is NP-hard \cite{PeqSouPed:15}. A special class of systems where the state bipartite graph $\B(\bA)$ has a perfect matching and every input can influence a single state (dedicated input) is discussed in \cite{PeqSouPed:15}. Note that under these assumptions the problem is not NP-hard. However, for the general case there are no known approximation results. Given $(\bA, \bB, \bC)$ finding the sparsest $\bK$ such that the closed-loop system has no SFMs is proved to be NP-hard in \cite{CarPeqAguKarPap:15}.

This paper focusses on minimum cost constrained input-output selection for generic arbitrary pole placement of structural systems. It is shown to be NP-hard in \cite{PeqKarPap:15}. This paper is motivated by \cite{PeqKarPap:15} where Pequito et.al investigated Problem~\ref{prob:one}  along with costs for $\bK$ on a class of systems whose graph is irreducible. For this class of systems Problem~\ref{prob:one} is not NP-hard. However, for general systems there are no known results. We address Problem~\ref{prob:one} in its full generality. Note that we do not assume cost on $\bK$. Unfortunately there do not exist polynomial algorithms for solving this unless P~=~NP. To this end, we propose an approximation algorithm for solving Problem~\ref{prob:one}. 

Our key contributions in this paper are threefold:\\
\noindent $\bullet$ We provide a polynomial time approximation algorithm that gives approximation ratio $6\,{\rm log\,}n$ for solving Problem~\ref{prob:one}.\\
\noindent $\bullet$ We prove that no polynomial time algorithm can achieve approximation
ratio $\frac{1}{4}{\rm log\,}n$. Thus the proposed algorithm is order optimal\\
\noindent $\bullet$ We show that the approximation can be much tighter in practical systems.\\
In the next section we detail our approach.
%%%%%%%%%%%%%%%%%%%%%%%%%%%%%%%%%%%%%%%%%%%%%%%%%%%%%%%%%%%%%555
\section{Approximating Minimum Cost Constrained Input-Output Selection Problem for Generic Arbitrary Pole Placement} \label{sec:minCC}
Our approach for solving Problem~\ref{prob:one} is to split the problem in to three sub-problems listed below: \\
\noindent$\bullet$ Minimum cost accessibility problem\\
\noindent$\bullet$ Minimum cost sensability problem\\
\noindent$\bullet$ Minimum cost disjoint cycle problem

Broadly, minimum cost accessibility (sensability, resp.) problem aims at finding minimum cost sub-collection of inputs (outputs, resp.) that cover all states. In minimum cost disjoint cycle problem, our aim is to find minimum cost sub-collection of inputs and outputs such that condition~b) is satisfied given that all chosen outputs connect to all chosen inputs (recall that $\bK_{ij} = \*$ for all $i,j$). For better readability and notational brevity we denote the structural system $(\bA, \bB, \bC)$ with feedback matrix $\bK$ and cost vectors $p_u, p_y$ as $(\bA, \bB, p_u)$ (without output) while discussing the accessibility problem  and as $(\bA, \bC, p_y)$ (without input) while discussing the sensability problem.

 Firstly we show that the minimum cost accessibility (sensability, resp.) problem is ``equivalent to" the weighted set cover problem. On account of the equivalence any algorithm for weighted set cover can be used for solving the minimum cost accessibility and sensability problems with the same performance guarantees and vice-versa. Weighted set cover problem is a well studied NP-hard problem \cite{CorLeiRivSte:01}. There exist approximation algorithms that give solution to the weighted set cover problem up to log factor in problem size \cite{Chv:79}. However, there also exist inapproximability result showing that it cannot be approximated up to a constant factor \cite{LunYan:94}. Thus, using the equivalence of the problems we provide an order optimal approximation algorithm to solve the minimum cost accessibility and sensability problems.

Then we show that the minimum cost disjoint cycle problem can be solved using a minimum cost perfect matching problem defined on $\B(\bA, \bB, \bC, \bK)$. Bipartite matching is also a well studied area and there exist polynomial time algorithm of complexity $O(\ell^3)$ that find minimum cost perfect matching in a bipartite graph with $\ell$ nodes on one side \cite{CorLeiRivSte:01}. Using the minimum cost perfect matching algorithm we provide a polynomial time algorithm to solve the minimum cost disjoint cycle problem optimally. Then we prove that combining the solutions to these sub-problems we can obtain an approximate solution to Problem~\ref{prob:one}.  Now we formally define and tackle each of these sub-problems separately in the following subsections.

\subsection{Solving Minimum Cost Accessibility Problem}\label{sec:access}
In this subsection, we establish a relation between the accessibility condition for structural controllability and the weighted set cover problem. Specifically, we show that when the inputs are constrained and each input is associated with a cost, then satisfying minimum cost accessibility condition is equivalent to solving a weighted set cover problem defined on the structural system.

Consider a structural system ($\bA, \bB$) and a cost vector $p_u$ denoted as ($\bA, \bB, p_u$). This system is said to satisfy the minimum cost accessibility condition if all the non-top linked SCCs in $\D(\bA)$ are covered using the least cost input set possible. That is, we need to find a set of inputs ${\I^\*_\A} \subseteq \{1,\ldots,m\}$ such that all state nodes are accessible in $\D(\bA, \bB_{\I^\*_\A}, \bC, \bK)$ and $p(\I^\*_\A) \leqslant p(\I_\A)$ for any $\I_\A \subseteq \{1,\ldots,m\}$ that satisfy accessibility of all state nodes in $\D(\bA, \bB_{\I_\A}, \bC, \bK)$. Specifically, we need to solve the following optimization: for any $\I \subseteq \{1,\ldots,m\}$, define $p(\I) = \sum_{i\in\I}p_u(i)$.
\begin{prob}\label{prob:access}
Given $(\bA, \bB, p_u)$, find $\I^\*_\A$

\[ \I^\*_\A ~\in~ \arg\min_{\10n \I_\A \subseteq \{1,\ldots,m\}} p(\I_\A), \] 
such that all state nodes are accessible in $\D(\bA, \bB_{\I_\A}, \bC, \bK)$.
\end{prob}

 We refer to Problem~\ref{prob:access} as {\it the minimum cost accessibility problem}. Before showing the equivalence between Problem~\ref{prob:access} and the weighted set cover problem, we first describe the weighted set cover problem for the sake of completeness. Weighted set cover problem is a well studied NP-hard problem  \cite{Chv:79}. Given a universe of $N$ elements $\U = \{ 1,2, \cdots, N\}$, a set of $r$ sets $\P = \{\S_1, \S_2, \cdots, \S_r \}$ with $\S_i \subset \U$ and $\bigcup_{i = 1}^r \S_i = \U$ and a weight function $w$ from $\P$ to the set of non-negative real numbers, weighted set cover problem consists of finding a set $\S^\* \subseteq \P$ such that $\cup_{\S_i \in \S^\*} \S_i = \U$ and $\sum_{\S_i \in \S^\*}w(i) \leqslant \sum_{\S_i \in \widetilde{\S}}w(i)$ for any $\widetilde{\S}$ that satisfies $\cup_{\S_i \in \widetilde{\S}} = \U$. Now we reduce Problem~\ref{prob:access} to an instance of the weighted set cover problem in polynomial time. 

\begin{algorithm}[t]
  \caption{Pseudo-code for reducing minimum cost accessibility problem to a weighted set cover problem
  \label{alg:access1}}
  \begin{algorithmic}
\State \textit {\bf Input:} Structural system $(\bA, \bB)$ and input cost vector $p_u$
\State \textit{\bf Output:} Input set $\I(\S)$ and cost $p(\I(\S))$
\end{algorithmic}
  \begin{algorithmic}[1]
\State Find all non-top linked SCCs in $\D(\bA)$, $\cN := \cN_1, \ldots,\cN_q$
  \State Define weighted set cover problem as follows:
  \State Universe $\U \leftarrow \{\cN_1, \ldots, \cN_q \}$\label{step:access_uni}
  \State Sets ${\S}_i \leftarrow\{\cN_j: \bB_{ri}=\* \mbox{~and~} x_r\in \cN_j \}$\label{step:access_set}
  \State Weights $w(i) \leftarrow p_u(i)$ for $i \in \{1,\ldots,m\}$\label{step:access_weight}  
 \State Given a cover $\S$ such that $\cup_{\S_i \in \S}\S_i \subseteq \U$, define: 
\State  Weight of the cover $w(\S) \leftarrow \sum_{\S_i \in \S}w_u(i)$\label{step:weightdef1}
\State Define $\I(\S) \leftarrow \{i:\S_i \in \S\}$ \label{step:IP1}
\State Cost of $\I(\S)$, $p(\I(\S)) \leftarrow \sum_{i \in \I(\S)}p_u(i)$ \label{step:cost1}
\end{algorithmic}
\end{algorithm}
 
The pseudo-code showing a reduction of Problem~\ref{prob:access} to an instance of weighted set cover problem is presented in Algorithm~\ref{alg:access1}. Given ($\bA, \bB, p_u$), we define a weighted set cover problem as follows: the universe $\U$ consists of all non-top linked SCCs $\{\cN_1,\ldots,\cN_q\}$ in $\D(\bA)$ (see Step~\ref{step:access_uni}). The Sets $\S_1, \ldots, \S_m$ is defined in such a way that set $\S_i$ consists of all non-top linked SCCs that are covered by the $i^{\rm th}$ input (see Step~\ref{step:access_set}). Further, for each set $\S_i$ we define weight $w(i)$ as shown in Step~\ref{step:access_weight}. 
Given a solution $\S$ to the weighted set cover problem, we define the associated weight $w(\S)$ as the sum of the weights of all sets selected under $\S$ (see Step~\ref{step:weightdef1}). Also, the indices of the sets selected in $\S$ is denoted as $\I(\S)$ and its cost is denoted as $p(\I(\S))$ as shown in Steps~\ref{step:IP1}~and~\ref{step:cost1} respectively. We denote an optimal solution to Problem~\ref{prob:access} as $\I^\*_\A$ and its cost as $p^\*_\A$. Also an optimal solution to the weighted set cover problem given in Algorithm~\ref{alg:access1} is denoted by $\S^\*_\A$ and its weight is denoted by $w^\*_\A$. Now we prove the following preliminary results.

\begin{lem}\label{lem:comp_1}
Consider any structural system $(\bA, \bB, \bC)$, feedback matrix $\bK$ and cost vectors $p_u, p_y$. Then, Algorithm~\ref{alg:access1} reduces Problem~\ref{prob:access} to a weighted set cover problem in $O(n^2)$ time. Moreover, for any cover $\S$, the set $\I(\S)$ and cost $p(\I(\S))$ given in Steps~\ref{step:IP1}~and~\ref{step:cost1} respectively can be obtained in $O(n)$ computations, where $n$ denotes the number of states in the system.
\end{lem}
\begin{proof}
Given state digraph $\D(\bA) = \D(V_X, E_X)$ all the non-top linked SCCs can be found in $O({\rm max}(|V_X|,|E_X|))$ computations. Here $|V_X| = n$ and $|E_X|$ is atmost $|V_X|^2$. Thus the reduction in Algorithm~\ref{alg:access1} has $O(n^2)$ computations. Also, given a cover $\S$ we can obtain $\I(\S)$ and $p(\I(\S))$ in linear time and this completes the proof. 
\end{proof}

\begin{lem}\label{lem:solution}
Consider any structural system $(\bA, \bB, \bC)$, feedback matrix $\bK$ and cost vectors $p_u, p_y$ and the corresponding weighted set cover problem obtained using Algorithm~\ref{alg:access1}. Let $\S$ be a feasible solution to the weighted set cover problem and $\I(\S)$ be the index set selected in Step~\ref{step:IP1}. Then, all states are accessible in $\D(\bA, \bB_{\I(\S)}, \bC, \bK)$ and $p(\I(\S)) = w(\S)$.
\end{lem}  
\begin{proof}
Given $\S$ is a feasible solution to the weighted set cover problem. Thus $\cup_{\S_i \in \S}\S_i = \U$. Hence, $\I(\S) = \{i:\S_i \in \S\}$ covers all the non-top linked SCCs in $\D(\bA)$. By Corollary~\ref{cor:acc} this implies that all states are accessible in $\D(\bA, \bB_{\I(\S)}, \bC, \bK)$. Now steps~\ref{step:access_weight},~\ref{step:weightdef1},~\ref{step:IP1}~and~\ref{step:cost1} of Algorithm~\ref{alg:access1} proves $p(\I(\S)) = w(\S)$. \qed
\end{proof}

In the following lemma we show that an $\epsilon$-approximation algorithm for the weighted set cover problem can be used to obtain an $\epsilon$-approximate solution to Problem~\ref{prob:access}.
\begin{lem}\label{lem:access_opt1}
Consider any structural system $(\bA, \bB, \bC)$, feedback matrix $\bK$ and cost vectors $p_u, p_y$ and the corresponding weighted set cover problem obtained using Algorithm~\ref{alg:access1}. Then, for $\epsilon > 1$, if $\S$ is an $\epsilon$-optimal solution to the weighted set cover problem, then $\I(\S)$ is an  $\epsilon$-optimal solution to the minimum cost accessibility problem. 
\end{lem}
\begin{proof}
The proof of this lemma is twofold: (i)~we show that an optimal solution $\S^\*_\A$ to the weighted set cover problem gives an optimal solution $\I^\*_\A$ to Problem~\ref{prob:access}, and (ii)~we show that if $w(\S) \leqslant \epsilon\,w^\*_\A$, then $ p(\I(\S)) \leqslant \epsilon\,p^\*_\A$. 

Given $\S^\*_\A$ is an optimal solution to the weighted set cover problem with cost $w^\*_\A$. For (i) we show that input set $\I(\S^\*_\A)$ selected under $\S^\*_\A$ is a minimum cost input set that satisfy the accessibility  of all states, i.e., all states are accessible in $\D(\bA, \bB_{\I(\S^\*_\A)}, \bC, \bK)$ and $p(\I(\S^\*_\A)) = p^\*_\A$. Since $\S^\*_\A$ is a solution to the weighted set cover problem, using Lemma~\ref{lem:solution} all states are accessible in $\D(\bA, \bB_{\I(\S^\*_\A)}, \bC, \bK)$. Thus $\I(\S^\*_\A)$ is a feasible solution to Problem~\ref{prob:access}. To prove minimality, we use a contradiction argument. Let us assume that $\S^\*_\A$ is an optimal solution to the weighted set cover problem but $\I(\S^\*_\A) = \{i:\S_i \in \S^\*_
\A \}$ is not a minimum cost input set that satisfy the accessibility condition. Then there exists $\I'_\A \subseteq \{1,\ldots,m\}$ such that all state nodes are accessible in $\D(\bA, \bB_{\I'_\A}, \bC, \bK)$ and $p(\I'_\A) < p(\I(\S^\*_\A))$. Note that for \mbox{$\S' = \{\S_i : i \in \I'_\A\}$}, $\cup_{\S_i \in \S'}\S_i = \U$. Using Lemma~\ref{lem:solution}, \mbox{$w(\S') < w^\*_\A$}. This gives a contradiction to the assumption that $\S^\*_\A$ is a minimal solution to the weighted set cover problem. This completes the proof of~(i). Now~(ii) follows from Lemma~\ref{lem:solution} and Step~\ref{step:access_weight} of Algorithm~\ref{alg:access1} and this completes the proof. 
\end{proof}
As an immediate consequence of the above result we can now show that approximation algorithm for minimum cost accessibility problem can be obtained from an approximation algorithm for the weighted set cover problem. 
\begin{thm}\label{th:access_1}
If there exists a polynomial time $\epsilon$-optimal algorithm for solving the weighted set cover problem, then there exists a polynomial time $\epsilon$-optimal algorithm for solving Problem~\ref{prob:access}. Thus, we can find a ${\rm log\,}\mu_{\rm max}$-optimal solution to Problem~\ref{prob:access}, where $\mu_{\rm max}$ is the maximum number of non-top linked SCCs covered by a single input.  
\end{thm}
\begin{proof}
From Lemma~\ref{lem:access_opt1}, a polynomial time $\epsilon$-optimal algorithm for solving the weighted set cover problem gives a polynomial time $\epsilon$-optimal algorithm for solving Problem~\ref{prob:access}. Now, using the greedy approximation algorithm for solving the weighted set cover problem given in \cite[pp.234]{Chv:79}, we can obtain a ${\rm log\,}\mu_{\rm max}$-optimal solution to Problem~\ref{prob:access}.  
\end{proof}

Note that through Algorithm~\ref{alg:access1} we have shown that any instance of Problem~\ref{prob:access} can be reduced in polynomial time to an instance of the weighted set cover problem. Now, we prove constant factor inapproximability of Problem~\ref{prob:access}. That is, there does not exist any polynomial time algorithm that give $\epsilon$-optimal solution to Problem~\ref{prob:access} for any $\epsilon > 1$. To achieve this we give a polynomial time reduction of the weighted set cover problem to an instance of Problem~\ref{prob:access} in Algorithm~\ref{alg:access2}. Using this, we will show that any polynomial time $\epsilon$-optimal algorithm for solving Problem~\ref{prob:access} can be used to get polynomial time $\epsilon$-optimal algorithm for the weighted set cover problem. Thus, since weighted set cover problem cannot be approximated up to constant factor, Problem~\ref{prob:access} also cannot be approximated up to constant factor.

\begin{algorithm}[t]
  \caption{Pseudo-code for reducing the weighted set cover problem to a minimum cost accessibility problem 
  \label{alg:access2}}
  \begin{algorithmic}
\State \textit {\bf Input:} Weighted set cover problem with universe $\U= \{1,\ldots, N\}$, sets $\P=\{\S_1,\ldots, \S_r\}$ and weight function $w$ 
\State \textit{\bf Output:} Structural system $(\bA, \bB)$ and input cost vector $p_u$ 
\end{algorithmic}
  \begin{algorithmic}[1]
  \State Define a minimum cost accessibility problem instance with $\bA \in \{0, \*\}^{N \times N}$, $\bB \in \{0, \* \}^{N \times r}$ and cost vector $p_u$ as follows:
  \State  $\bA_{ij} \leftarrow \begin{cases}
    \*$, for $i=j, \label{step:access_A}\\
    0, \mbox{~otherwise~}.
  \end{cases} $
   \State  $\bB_{ij} \leftarrow \begin{cases}
    \*$, for $i \in \S_j, \label{step:access_B}\\
    0, \mbox{~otherwise~}.
  \end{cases} $ 
  \State $p_u(i) \leftarrow w(i)$, for $i \in \{1,\ldots,r\}$\label{step:access_ip}
 \State Given a set $\I$ such that all states are accessible in $\D(\bA, \bB_{\I}, \bC, \bK)$, define: 
\State  Cost of the set $p(\I) \leftarrow \sum_{i \in \I}p_u(i)$,\label{step:weightdef2}
\State Define $\S(\I) \leftarrow \{\S_i: i \in \I\}$, \label{step:IP2}
\State Weight of $\S(\I)$, $w(\S(\I)) \leftarrow \sum_{\S_i \in \S(\I)}w(i)$ \label{step:cost2}.
\end{algorithmic}
\end{algorithm}
The pseudo-code showing a reduction of the weighted set cover problem to an instance of Problem~\ref{prob:access} is presented in Algorithm~\ref{alg:access2}. Given $\U, \P$ and $w$, we reduce the weighted set cover problem to an instance of the minimum cost accessibility problem. Here, $\bA$ is a diagonal $N \times N$ matrix with all diagonal entries $\*$'s (see Step~\ref{step:access_A}). Now, $\bB$ is defined in such a way that its $j^{\rm th}$ column corresponds to the set $\S_j$ (see Step~\ref{step:access_B}) and cost of $j^{\rm th}$ input is same as the weight $w(j)$ of $\S_j$ (see Step~\ref{step:access_ip}).    
Given a solution $\I$ to the accessibility problem, we define the associated cost $p(\I)$, the sets selected $\S(\I)$ and its weight $w(\S(\I))$ as shown in Steps~\ref{step:weightdef2},\ref{step:IP2}~and~\ref{step:cost2} respectively. We denote an optimal solution to the set cover problem in Algorithm~\ref{alg:access2} as $\S^\*$ and its weight as $w^\*$. Now we prove the following preliminary results.

\begin{lem}\label{lem:comp_2}
Consider any weighted set cover problem with universe $\U$, set $\P$ and weight $w$. Let $|\U| = N$. Then, Algorithm~\ref{alg:access2} reduces the weighted set cover problem to Problem~\ref{prob:access} in $O(N^2)$ computations. Moreover, for any set $\I$, the cover $\S(\I)$ and weight $w(\S(\I))$ given in Steps~\ref{step:IP2}~and~\ref{step:cost2} respectively can be obtained in $O(N)$ computations.
\end{lem}
\begin{proof}
Given any weighted set cover problem $\U, \P, w$, matrices $\bA$, $\bB$ can be found in $O(N)$, $O(N^2)$ computations respectively. Also, cost vector $p_u$ can be found in linear time. Thus the reduction of the set cover problem to an instance of Problem~\ref{prob:access} given in Algorithm~\ref{alg:access2} has $O(N^2)$ computations. Also, given a set $\I$ we can obtain $\S(\I)$ and $w(\S(\I))$  in linear time and this completes the proof. 
\end{proof}

\begin{lem}\label{lem:solution2}
Consider any weighted set cover problem given by $\U, \P, w$ and the corresponding structural system obtained using Algorithm~\ref{alg:access2}. Let $\I$ be a feasible solution to Problem~\ref{prob:access} and $\S(\I)$ consists of the sets selected under $\I$. Then, $\S(\I)$ covers $\U$ and $w(\S(\I)) = p(\I) $.
\end{lem}  
\begin{proof}
Given $\I$ is a feasible solution to Problem~\ref{prob:access}. Thus all states are accessible in 
$\D(\bA, \bB_\I, \bC, \bK)$. This implies for $\S(\I) = \{\S_i: i \in \I\}$, $\cup_{\S_i \in \S(\I)}\S_i = \U$. Thus by Corollary~\ref{cor:acc} $\S(\I)$ covers $\U$.  Now Steps~\ref{step:access_ip},~\ref{step:weightdef2},~\ref{step:IP2}~and~\ref{step:cost2} of Algorithm~\ref{alg:access2} gives $w(\S(\I)) = p(\I) $.
\end{proof}
In the following lemma we show that an $\epsilon$-approximation algorithm for Problem~\ref{prob:access} can be used to obtain an $\epsilon$-approximate solution to the weighted set cover problem.
\begin{lem}\label{lem:access_opt2}
Consider any weighted set cover problem and the corresponding structural system $(\bA, \bB, p_u)$ obtained using Algorithm~\ref{alg:access2}. For $\epsilon > 1$, if $\I$ is an $\epsilon$-optimal solution to the minimum cost accessibility problem, then $\S(\I)$ is an  $\epsilon$-optimal solution to the weighted set cover problem. 
\end{lem}
\begin{proof}
The proof of this lemma is twofold: (i)~we show that an optimal solution $\I^\*_\A$ to Problem~\ref{prob:access} gives an optimal solution $\S^\*_\A$ to the weighted set cover problem, and (ii)~we show that, if $p(\I) \leqslant \epsilon\,p^\*_\A$, then $w(\S(\I)) \leqslant \epsilon\,w^\*_\A$.

For proving (i)~we assume that $\I^\*_\A$ is an optimal solution to Problem~\ref{prob:access} and then prove that $\S(\I^\*_\A)$ is an optimal solution to the weighted set cover problem, i.e, $\cup_{\S_i \in \S(\I^\*_\A)} = \U$ and  $w(\S(\I^\*_\A)) = w^\*_\A$. Given $\I^\*_\A$ is an optimal solution to Problem~\ref{prob:access}. Thus all states are accessible in $\D(\bA, \bB_{\I^\*_\A}, \bC, \bK)$. Hence, by Lemma~\ref{lem:solution2}, $\S(\I^\*_\A)$ is a feasible solution to the weighted set cover problem. Now we prove optimality using a contradiction argument. Let $\I^\*_\A$ is an optimal solution to Problem~\ref{prob:access}, but $\S(\I^\*_\A)$ is not an optimal solution to the weighted set cover problem. Then there exists $\widetilde{\S} \subset \{\S_1,\ldots,\S_r\}$ such that $\cup_{\S_i \in \widetilde{\S}}\S_i = \U$ and $ w(\widetilde{\S}) <  w(\S(\I^\*_\A))$. Then $\widetilde{\I} = \{i:\S_i \in \widetilde{\S}\}$ covers all the non-top linked SCCs in $\D(\bA)$. Also, from Lemma~\ref{lem:solution}, $p(\widetilde{\I}) < p^\*_\A$. This gives a contradiction to the assumption that $\I^\*_\A$ is a minimum cost input set  that satisfies accessibility condition. This completes the proof of~(i). Now (ii)~follows directly from Lemma~\ref{lem:solution2} and Step~\ref{step:access_ip} of Algorithm~\ref{alg:access2}. This completes the proof.
\end{proof}
Lemmas~\ref{lem:access_opt1}~and~\ref{lem:access_opt2} prove the equivalence of Problem~\ref{prob:access} and the weighted set cover problem. There are  no polynomial algorithms for solving weighted set cover problem unless P~=~NP. However, there exist various approximation algorithms that find approximate solution to the weighted set cover problem. Specifically, the greedy approximation algorithm given in \cite{Chv:79} gives a log$\,d$ approximation, where $d$ is the cardinality of the largest set $\S_i$ in $\P$. 
In addition to this, we also know strong negative approximability result for the set cover problem. The set cover problem is a special case of weighted set cover problem, where all weights are non-zero and uniform. Thus the inapproximability result of the set cover problem applies to the weighted set cover problem also.
\begin{prop}\cite[Theorem 4.4]{Fei:98}\label{prop:inapprox}
If there is some $\epsilon > 0$ such that a polynomial time algorithm can approximate the set cover problem within $(1-\epsilon)\,{\rm log}\,L$, then $NP \subset NTIME(L^{{\rm log\, log\,}\,L})$, where $L$ denotes the number of items in the universe.
\end{prop}
Using Lemma~\ref{lem:access_opt2} and Proposition~\ref{prop:inapprox} we can now show that inapproximability result of the weighted set cover problem implies inapproximability result of Problem~\ref{prob:access}.
\begin{thm}\label{th:access_2}
If there does not exist a polynomial time $\epsilon$-optimal algorithm for solving the weighted set cover problem, then there does not exist a polynomial time $\epsilon$-optimal algorithm for solving Problem~\ref{prob:access}. Moreover, there does not exist a polynomial time algorithm that can approximate Problem~\ref{prob:access} to factor $(1-o(1))\,{\rm log\,}q$, where $q$ denotes the number of non-top linked SCCs in $\D(\bA)$.
\end{thm}
\begin{proof}
From Lemma~\ref{lem:access_opt2}, a polynomial time $\epsilon$-optimal algorithm for solving Problem~\ref{prob:access} gives a polynomial time $\epsilon$-optimal algorithm for solving the weighted set cover problem. Now, from Proposition~\ref{prop:inapprox} weighted set cover problem cannot be approximated up to factor $(1-O(1))\,{\rm log\,}N$, where $N$ is the cardinality of the universe. The weighted set cover reduction of Problem~\ref{prob:access} has $|\U| = q$. Thus Problem~\ref{prob:access} cannot be approximated to factor $(1-o(1))\,{\rm log\,}q$.  
\end{proof}

This shows the hardness of the problem. The number of non-top linked SCCs is atmost $n$. This happens when each state is decoupled. However, in practical cases the states are not decoupled. The more connected the graph is, the number of non-top linked SCCs are less. In such cases the above result gives a tighter bound. In the following sub-section we discuss briefly about the minimum cost sensability problem.
%%%%%%%%%%%%%%%%%%%%%%%%%%%%%%%%%%%%%%%%%%%%%%%%%%%%%%%%%%%%%%%%%%%5
\subsection{Solving Minimum Cost Sensability Problem}\label{sec:dil}
In this section, we establish a relation between the sensability condition for structural observability and a set cover problem. Specifically, we show that when the outputs are constrained and each output is associated with a cost, then satisfying minimum cost sensability condition is equivalent to solving a weighted set cover problem defined on the structural system.

Consider a structural system ($\bA, \bC$) and a cost vector $p_y$ denoted as ($\bA, \bC, p_y$). This system is said to satisfy the minimum cost sensability condition if all the non-bottom linked SCCs in $\D(\bA)$ are covered by the least cost output set possible. 
That is, we need to find a set of outputs $\J^\*_\A \subseteq \{1,\ldots,p\}$ such that all state nodes are sensable in $\D(\bA,\bB, \bC_{\J^\*_\A}, \bK)$ and $p(\J^\*_\A) \leqslant p(\J_\A)$ for any $\J_\A \subseteq \{1,\ldots,p\}$ that satisfy sensability of all state nodes in $\D(\bA, \bB, \bC_{\J_\A}, \bK)$. We refer to the above problem as {\it the minimum cost sensability problem}.

 However, because of duality between controllability and observability solving minimum cost sensability problem is equivalent to solving minimum cost accessability problem of the structural system $(\bA^T, \bC^T, p_y)$. Thus the weighted set cover reformulation of Problem~\ref{prob:access} for $(\bA^T, \bC^T, p_y)$ solves the minimum cost sensability problem of $(\bA, \bC, p_y)$. Hence the following result immediately follows from the analysis done in the previous sub-section. 

\begin{cor}\label{cor:setcover2}
Consider a structurally observable system ($\bA, \bC, p_y$). We can find a ${\rm log\,}\eta_{\rm max}$-optimal solution to the minimum cost sensability problem, where $\eta_{\rm max}$ is the maximum number of non-bottom linked SCCs covered by a single output. Also, there does not exist polynomial time algorithm that can approximate minimum cost sensability problem to factor $(1-o(1)))\,{\rm log\,}k$, where $k$ is the number of non-bottom linked SCCs in $\D(\bA)$. 
\end{cor}
 Now we will find a relation between minimum cost disjoint cycle condition and a bipartite matching problem. 
%%%%%%%%%%%%%%%%%%%%%%%%%%%%%%%%%%%%%%%%%%%%%%%%%%%%%%%%%%%%%%%%5
\subsection{Solving Minimum Cost Disjoint Cycle Problem}
In this subsection we establish a relation between disjoint cycle condition and perfect matching problem. Specifically, we show that when the inputs and outputs are constrained and each input and output are associated with costs, then satisfying disjoint cycle condition using a minimum cost input-output set is equivalent to solving a minimum cost perfect matching problem on a bipartite graph defined on the structural system. 

A structural system ($\bA, \bB, \bC$) with feedback matrix $\bK$ and cost vectors $p_u, p_y$ is said to satisfy the minimum cost disjoint cycle condition if all state vertices are spanned by disjoint union of cycles in the system digraph by using the least possible cost input-output set. That is, we need to find an input set $\I^\*_\C \subseteq \{1,\ldots,m\}$ and an output set $\J^\*_\C \subseteq \{1,\ldots,p\}$ such that all $x_i$'s are spanned by disjoint cycles in $\D(\bA, \bB_{\I^\*_\C}, \bC_{\J^\*_\C}, \bK_{(\I^\*_\C \times \J^\*_\C)})$ and $p(\I^\*_\C) + p(\J^\*_\C) \leqslant p(\I) + p(\J)$ for any $\I \subseteq \{1,\ldots,m \}$ and $\J \subseteq \{1,\ldots,p \}$ that satisfy disjoint cycle condition in $\D(\bA, \bB_{\I}, \bC_{\J}, \bK_{(\I \times \J)})$. Specifically, we need to solve the following optimization problem.

\begin{prob}\label{prob:cycle}
Given $(\bA, \bB, \bC, \bK)$ and cost vectors $p_u$ and $p_y$, find
\[
(\I^\*_\C, \J^\*_\C) ~\in~ \arg\min_{\10n \substack{\I_\C \subseteq \{1,\ldots,m\} \\ \J_\C \subseteq \{1,\ldots,p\}}} p(\I_\C, \J_\C),
\]
such that all $x_i$'s lie in finite disjoint union of cycles in $\D(\bA, \bB_{\I_\C}, \bC_{\J_\C}, \bK_{(\I_\C \times \J_\C)})$.
\end{prob}

We refer to Problem~\ref{prob:cycle} as {\it the minimum cost disjoint cycle problem}. Now we reduce the minimum cost disjoint cycle problem to a minimum cost perfect matching problem.  

\begin{algorithm}[t]
  \caption{Pseudo-code for reducing minimum cost disjoint cycle problem to a minimum cost perfect matching problem
  \label{alg:cycle}}
  \begin{algorithmic}
\State \textit {\bf Input:} Structural system $(\bA, \bB, \bC, \bK)$ and cost vectors $p_u, p_y$ 
\State \textit{\bf Output:} Input-output set $(\I(M_\C), \J(M_\C))$ and cost $p(\I(M_\C), \J(M_\C))$
\end{algorithmic}
  \begin{algorithmic}[1]
  \State Construct the bipartite graph $\B(\bA, \bB, \bC, \bK)$
  \State For $e \in \E_X \cup \E_U \cup \E_Y \cup \E_K \cup \E_{\mathbb{U}} \cup \E_{\mathbb{Y}}$ define:
  \State Cost, $ c(e) \leftarrow 
\begin{cases}
p_u(i) + p_y(j), {~\rm for~} e = (u'_i,y_j) \in \E_K,\\
0, ~~   {\rm otherwise}.
\end{cases}$ \label{step:bip_cost}
   \State  Find minimum cost perfect matching of $\B(\bA, \bB, \bC, \bK)$ under cost $c$, say $M_\C$
\State  Cost of $M_\C$, $c(M_\C) \leftarrow \sum_{e \in M_\C}c(e)$\label{step:cost_match}
\State Input index set selected under $M_\C$, $\I(M_\C) \leftarrow \{i:(x'_j,u_i) \in M_\C\}$ \label{step:ip_match}
\State Input cost $p(\I(M_\C)) \leftarrow \sum_{i \in \I(M_\C)}p_u(i)$\label{step:ip_cost_match}
\State Output index set selected under $M_\C$, $\J(M_\C) \leftarrow \{j:(y'_j,x_i) \in M_\C\}$ \label{step:op_match}
\State Output cost $p(\J(M_\C)) \leftarrow  \sum_{j \in \J(M_\C)}p_y(j)$\label{step:op_cost_match}.
\end{algorithmic}
\end{algorithm}

Pseudo-code for reducing the minimum cost disjoint cycle problem to a minimum cost perfect matching problem is presented in Algorithm~\ref{alg:cycle}. The bipartite graph $\B(\bA, \bB, \bC, \bK)$ constructed in \cite{PeqKarPap:15} for a special case is used here to guarantee condition~b) in Proposition~\ref{prop:SFM1} for a general case. Given the bipartite graph $\B(\bA, \bB, \bC, \bK)$ and the cost function $c$ defined as in Step~\ref{step:bip_cost}, we find a perfect matching $M_\C$. On obtaining a perfect matching $M_\C$, we define the associated cost $c(M_\C)$ as the sum of the costs of edges that are present in $M_\C$ (see Step~\ref{step:cost_match}). The input index set selected under $M_\C$ defined as $\I(M_\C)$ is the set of indices of $u_i$'s that are connected to some state vertices in $M_\C$ (see Step~\ref{step:ip_match}) and its cost is defined as $p(\I(M_\C))$ (see Step~\ref{step:ip_cost_match}). Now, the output index set selected under $M_\C$ defined as $\J(M_\C)$ consists of indices of all outputs $y_j$'s that are connected to some state vertices in $M_\C$ (see Step~\ref{step:op_match}) and its cost is defined as $p(\J(M_\C))$ (see Step~\ref{step:op_cost_match}).

We denote an optimal solution to the minimum cost perfect matching problem as $M^\*$ and the optimal cost as $c^\*$. Also, an optimal  solution to Problem~\ref{alg:cycle} is denoted as $(\I^\*_\C, \J^\*_\C)$ and the optimal input-output cost is denoted as $(p^\*_{\C u} + p^\*_{\C y})$. We prove the following theorem to give a necessary and sufficient condition for condition~b) in Proposition~\ref{prop:SFM1} for the sake of completeness.

\begin{thm}\label{th:disj}
Consider a structural system $(\bA, \bB, \bC)$ with feedback matrix $\bK$. Then, the bipartite graph $\B(\bA, \bB, \bC, \bK)$ has a perfect matching if and only if all states are spanned by disjoint union of cycles in $\D(\bA, \bB, \bC, \bK)$.
\end{thm}
\begin{proof}
\noindent{\bf Only-if part:} We assume that the bipartite graph $\B(\bA, \bB, \bC, \bK)$ has a perfect matching and prove that all state nodes are spanned by disjoint union of cycles in $\D(\bA, \bB, \bC, \bK)$. Let $M$ be a perfect matching in $\B(\bA, \bB, \bC, \bK)$. Let $\E' = \{(u'_i, u_i), (y'_j, y_j)\}\in M$ for $i \in \{1,\ldots,m\}$ and $j \in \{1,\ldots,p\}$.  Thus edges in $M \setminus \E'$ correspond to edges in $\D(\bA, \bB, \bC, \bK)$ such that there exist one incoming edge and one outgoing edge corresponding to every vertex in $\D(\bA, \bB, \bC, \bK)$ except nodes $u_i$'s and $y_j$'s that has edges in $\E'$. Since corresponding to edges in $M \setminus \E'$ every vertex has both in-degree and out-degree one, these edges corresponds to disjoint cycles in $\D(\bA, \bB, \bC, \bK)$. Note that all state vertices lie in $M \setminus \E'$. Hence, all $x_i$'s are spanned by disjoint union of cycles. This completes the proof of only-if part.

\noindent{\bf If part:} We assume that there exist disjoint union of cycles that span all state nodes in $\D(\bA, \bB, \bC, \bK)$ and prove that there exists a perfect matching in $\B(\bA, \bB, \bC, \bK)$. Since the cycles are disjoint, each node in it has one incoming edge and one outgoing edge. Each edge in the cycle corresponds to an edge in the bipartite graph. Vertices in $\D(\bA, \bB, \bC, \bK)$ that are not covered by these cycles will belong to the set of input and output nodes only. For such nodes there exist edges $(u'_i, u_i)$ for all $i \in \{1,\ldots,m\}$ and $(y'_j, y_j)$ for all $j \in \{1,\ldots,p\}$ in $\B(\bA, \bB, \bC, \bK)$. These edges along with the cycle edges results in a perfect matching. This completes the proof.
\end{proof}

\begin{lem}\label{lem:disj_soln} 
Let $M_\C$ be a perfect matching in $\B(\bA, \bB, \bC, \bK)$ and $\I(M_\C), \J(M_\C)$ denote the index set of inputs and index set of outputs selected under $M_\C$ respectively. Then, all $x_i$'s lie in disjoint cycles in $\D(\bA, \bB_{\I(M_\C)}, \bC_{\J(M_\C)}, \bK_{(\I(M_\C) \times \J(M_\C))})$ and $p(\I(M_\C))+ p(\J(M_\C))= p(\I(M_\C), \J(M_\C)) = c(M_\C)$.
\end{lem}

\begin{proof}
Given $M_\C$ is a perfect matching in the bipartite graph $\B(\bA, \bB, \bC, \bK)$ with cost function $c$. Using Theorem~\ref{th:disj}, there exist disjoint cycles that cover all state nodes in $\D(\bA, \bB_{\I(M_\C)}, \bC_{\J(M_\C)}, \bK_{(\I(M_\C) \times \J(M_\C))})$. Now, Step~\ref{step:bip_cost} and Steps~\ref{step:ip_match}~to~\ref{step:op_cost_match} in Algorithm~\ref{alg:cycle} gives $p(\I(M_\C)) + p(\J(M_\C)) = c(M_\C)$.\qed
\end{proof}

Now we prove that minimum cost perfect matching problem on $\B(\bA, \bB, \bC, \bK)$ with cost $c$  can be used to solve the minimum cost disjoint cycle problem.

\begin{thm}\label{th:mincost_matching}
Consider a structural system $(\bA, \bB, \bC)$ with feedback matrix $\bK$ and cost vectors $p_u, p_y$. Let $(\I^\*_\C, \J_\C^\*)$ be an optimal solution  to Problem~\ref{prob:cycle} and $p(\I^\*_\C, \J_\C^\*)$ be the optimal cost of Problem~\ref{prob:cycle}. Let $c^\*$ is the optimal cost of the minimum cost perfect matching problem on $\B(\bA, \bB, \bC, \bK)$. Then, $c^\* = p(\I^\*_\C, \J^\*_\C)$. Moreover, the input index set and output index set selected under Algorithm~\ref{alg:cycle} provide an optimal solution to Problem~\ref{prob:cycle}.
\end{thm}
\begin{proof}
Given $(\I^\*_\C, \J_\C^\*)$ is an optimal solution to Problem~\ref{prob:cycle}. Then, from Theorem~\ref{th:disj} there exists a perfect matching in $\B(\bA, \bB_{\I^\*_\C}, \bC_{\J^\*_\C}, \bK_{(\I^\*_\C \times \J^\*_\C)})$. Let $M$ be an optimum matching in $\B(\bA, \bB_{\I^\*_\C}, \bC_{\J^\*_\C}, \bK_{(\I^\*_\C \times \J^\*_\C)})$. Then, $c(M) \leqslant p(\I^\*_\C, \J_\C^\*)$. Note that $\widetilde{M} = M \cup \{(u'_i, u_i):i \notin \I^\*_\C\} \cup \{(y'_j, y_j):j \notin \J^\*_\C\}$ is an optimum matching in $\B(\bA, \bB, \bC, \bK)$. Also $c(\widetilde{M}) = c(M)$. Thus $c(\widetilde{M}) = c^\* \leqslant p(\I^\*_\C, \J_\C^\*)$.

Now let $M^\*$ is an optimal solution to the minimum cost perfect matching problem in $\B(\bA, \bB, \bC, \bK)$. Then $c(M^\*) = c^\*$. By Theorem~\ref{th:disj} there exists disjoint cycles whose union span all $x_i$'s in $\D(\bA, \bB_{\I(M^\*)}, \bC_{\J(M^\*)}, \bK_{(\I(M^\*) \times \J(M^\*))})$. Let the input-output set used in these cycles are $(\I, \J)$. Now $p(\I, \J) \leqslant c^\*$. Also, $ p(\I^\*_\C, \J_\C^\*) \leqslant p(\I, \J)$. Thus $p(\I^\*_\C, \J_\C^\*) \leqslant c^\*$. Combining both, we get $p(\I^\*_\C, \J_\C^\*) = c^\*$.  

Now we assume that $M^\*$ is an optimal solution to the minimum cost perfect matching problem with cost $c^\*$ and then show that input-output set $(\I(M^\*), \J(M^\*))$ selected under $M^\*$ is an optimal solution to Problem~\ref{prob:cycle}, i.e., all $x_i$'s lie in disjoint union of cycles in $\D(\bA, \bB_{\I(M^\*)}, \bC_{\J(M^\*)}, \bK_{(\I(M^\*) \times \J(M^\*))})$ and $p(\I(M^\*), \J(M^\*))$ $= p(\I^\*_\C, \J_\C^\*)$.

Since $M^\*$ is a solution to the minimum cost perfect matching problem, by Lemma~\ref{lem:disj_soln} there are disjoint cycles in $\D(\bA, \bB_{\I(M^\*)}, \bC_{\J(M^\*)}, \bK_{(\I(M^\*), \J(M^\*))})$ such that all state nodes lie in their union. Thus $(\I(M^\*), \J(M^\*))$ is a feasible solution to Problem~\ref{prob:cycle}. To prove minimality we use a contradiction argument. Let us assume that $M^\*$ is an optimal matching but $(\I(M^\*), \J(M^\*))$ is not an optimal solution to Problem~\ref{prob:cycle}. Then there exists $\I'_\C \subset \{1,\ldots,m\}$ and $\J'_\C \subset \{1,\ldots,p\}$ that satisfy the disjoint cycle condition in $\D(\bA, \bB_{\I'_\C}, \bC_{\J'_\C}, \bK_{(\I'_\C \times \J'_\C)})$ and $p(\I'_\C,\J'_\C) < p(\I(M^\*), \J(M^\*))$. Then by Theorem~\ref{th:disj} there exists a perfect matching $M'$ such that $\I(M') = \I'_\C$ and $\J(M') = \J'_\C$. Using Lemma~\ref{lem:disj_soln}, \mbox{$c(M') < c^\*$}. This gives a contradiction to the assumption that $M^\*$ is an optimal matching. This completes the proof. 
\end{proof}

Hence, an optimal solution $M^\*$ to the minimum cost perfect matching problem gives a minimum cost input-output set $(\I^\*_\C, \J^\*_\C)$ that satisfies the disjoint cycle condition. There exist efficient polynomial time algorithms to solve the minimum cost perfect matching problem \cite{CorLeiRivSte:01}. Thus using these algorithms we can solve Problem~\ref{prob:cycle} optimally in polynomial time.  In the next section we give an approximation algorithm to solve Problem~\ref{prob:one}.
%%%%%%%%%%%%%%%%%%%%%%%%%%%%%%%%%%%%%%%%%%%%%%%%%%%%%%%%%%%%%%%%%%%5
\section{Approximating Constrained Input-Output Selection for Generic Arbitrary Pole Placement}\label{sec:main}
In this section we give a polynomial time approximation algorithm for solving Problem~\ref{prob:one}. We propose a three stage algorithm for solving Problem~\ref{prob:one}. The pseudo-code for the proposed algorithm is given in Algorithm~\ref{alg:threestage}. In the first stage of Algorithm~\ref{alg:threestage} we solve a weighted set cover problem defined on the structural system $(\bA, \bB, p_u)$ using a greedy approximation algorithm given in \cite{Chv:79} to obtain an approximate solution to the minimum cost accessibility problem. We define the input index set selected under its solution as $\hat{\I}^\*_\A$ (see Step~\ref{step:soln1}). Subsequently, in stage two we solve a weighted set cover problem  defined on the structural system $(\bA, \bC, p_y)$ to approximate the minimum cost sensability problem. We define the output index set selected under its solution as $\hat{\J}^\*_\A$ (see Step~\ref{step:soln2}). In the third stage of the algorithm a minimum cost perfect matching problem is solved on $\B(\bA, \bB, \bC, \bK)$ with cost function $c$. We define the input-output index set selected under solution to this problem as $(\I^\*_\C, \J^\*_\C)$ (see Step~\ref{step:soln3}). In one of our main result we prove that $(\hat{\I}^\*_\A \cup \I^\*_\C, \hat{\J}^\*_\A \cup \J^\*_\C)$ is an approximate solution to Problem~\ref{prob:one}. Firstly, we prove the following preliminary result.
\begin{algorithm}[t]
  \caption{Pseudo-code for solving minimum cost accessibility, sensability and disjoint cycle problems
  \label{alg:threestage}}
  \begin{algorithmic}
\State \textit {\bf Input:} Structural system $(\bA, \bB, \bC)$, feedback matrix $\bK$, input cost vector $p_u$, output cost vector $p_y$
\State \textit{\bf Output:} Input set $\I_a$ and output set $\J_a$
\end{algorithmic}
  \begin{algorithmic}[1]
  \State Find approximate solution to the minimum cost accessibility problem on $(\bA, \bB, p_u)$, say $\hat{\I}^\*_\A$ \label{step:soln1}  
  \State Find approximate solution to the minimum cost sensability problem on $(\bA, \bC, p_y)$, say $\hat{\J}^\*_\A$\label{step:soln2}
\State Find optimal solution to the minimum cost disjoint cycle problem on $\B(\bA, \bB, \bC, \bK)$ under cost function $c$, say $(\I^\*_\C, \J^\*_\C)$ \label{step:soln3}
\State $\I_a \leftarrow \hat{\I}^\*_\A \cup \I^\*_\C$ \label{step:soln5}
\State $\J_a \leftarrow \hat{\J}^\*_\A \cup \J^\*_\C$ \label{step:soln6}
\end{algorithmic}
\end{algorithm}

\begin{lem}\label{lem:SCC}
Let $\D(\bA)$ denote the state digraph of a structural system. Then, either one of the following happens:\\
\noindent $\bullet$ an SCC in $\D(\bA)$ is both non-top linked and non-bottom linked, \\
\noindent $\bullet$ an SCC in $\D(\bA)$ lies in a path starting at some non-top linked SCC and ending at some non-bottom linked SCC.
\end{lem}

\begin{proof}
Consider the Directed Acyclic Graph (DAG) whose vertices are the SCCs in $\D(\bA)$ and an edge exists between two nodes if and only if there exists an edge connecting two states in those respective SCCs in $\D(\bA)$. The nodes in the DAG are of two types: (i)~isolated, and (ii)~has an incoming and/or outgoing edge. In case~(i) the corresponding SCC is both non-top linked and non-bottom linked. In case~(ii) it has either an incoming edge or an outgoing edge or both. Thus those SCCs lie in some path from some non-top linked SCC to some non-bottom linked SCC since the DAG is acyclic. This completes the proof.  
\end{proof}
Now we prove our main result.

\noindent{\it Proof~of~Theorem~\ref{th:main}}:
Given $(\I_a, \J_a)$ is an output of Algorithm~\ref{alg:threestage}. Hence, all states are accessible in $\D(\bA, \bB_{\I_a}, \bC, \bK)$ and states are sensable in $\D(\bA, \bB, \bC_{\J_a}, \bK)$. Thus, in $\D(\bA, \bB_{\I_a}, \bC_{\J_a}, \bK_{(\I_a \times \J_a)})$ all states are both accessible and sensable. Consider an arbitrary state $x$ which belongs to some SCC $\mathcal{N}$. By Lemma~\ref{lem:SCC}, $\mathcal{N}$ lies on some path from a non-top linked SCC, say $\cN$, to a non-bottom linked SCC, say $\cNn$, in the SCC DAG. Since $U = \{u_i:i \in \I_a \}$ are enough for accessibility, there exists $u \in U$ such that $u$ covers $\cN$. Similarly, since $Y = \{y_j: j \in \J_a \}$ are enough for sensability there exists $y \in Y$ such that $y$ covers $\cNn$. Since $\bK$ is complete $(y,u)$ belong to $\D(\bA, \bB_{\I_a}, \bC_{\J_a}, \bK_{(\I_a \times \J_a)})$. Thus in this digraph all states in all the SCCs of $\D(\bA)$ that lie in the path from $\cN$ to $\cNn$ now belong to a single SCC in $\D(\bA, \bB_{\I_a}, \bC_{\J_a}, \bK_{(\I_a \times \J_a)})$ which has edge $(y,u)$. Thus $x$ belongs to an SCC in $\D(\bA, \bB_{\I_a}, \bC_{\J_a}, \bK_{(\I_a \times \J_a)})$ with a $(y,u)$ edge. Since $x$ is arbitrary condition~a) in Proposition~\ref{prop:SFM1} follows.
Since $(\I_a, \J_a)$ is an output of Algorithm~\ref{alg:threestage}, by Theorem~\ref{th:disj} there exists disjoint cycles that cover all state nodes using inputs whose indices are in $\I_a$ and outputs whose indices are in $\J_a$. Thus $(\I_a, \J_a)$ satisfies condition~b) in Proposition~\ref{prop:SFM1}. Thus $(\I_a, \J_a) \in \K$. This completes the proof of~i).

Let $\I^\*_\A$ and $\J^\*_\A$ are optimal solutions to the minimum cost accessibility problem and minimum cost sensability problem respectively. Given $(\I_a, \J_a)$ is an output of Algorithm~\ref{alg:threestage}. Let $\I_a = \hat{\I}^\*_\A \cup \I^\*_\C$, where $\hat{\I}^\*_\A$ is an $\epsilon_1$-optimal solution to the minimum cost accessibility problem and $\I^\*_\C$ is a minimum cost set that satisfy the disjoint cycle condition. Similarly, $\J_a = \hat{\J}^\*_\A \cup \J^\*_\C$, where $\hat{\J}^\*_\A$ is an $\epsilon_2$-optimal solution to the minimum cost sensability problem and $\J^\*_\C$ is a minimum cost set that satisfy the disjoint cycle condition.  Now by Theorem~\ref{th:access_1}, $\epsilon_1 \leqslant {\rm log}\,\mu_{\rm max}$ and by Corollary~\ref{cor:setcover2}, $\epsilon_2 \leqslant {\rm log}\,\eta_{\rm max}$. Since $(\I^\*, \J^\*)$ is an optimal solution to Problem~\ref{prob:one} its cost is atleast the cost of satisfying the two conditions in Proposition~\ref{prop:SFM1} separately. This give Equations~\eqref{eq:one}~and~\eqref{eq:three}.
\begin{eqnarray}
p(\I^\*,\J^\*) & \geqslant & p(\I^\*_\A) +  p(\J^\*_\A), \label{eq:one}\\
p(\I^\*,\J^\*) & \geqslant & p(\I^\*_\C,\J^\*_\C), \label{eq:three}\\
2p(\I^\*,\J^\*) & \geqslant & p(\I^\*_\A) + p(\J^\*_\A)+p(\I^\*_\C) + p(\J^\*_\C), \nonumber \\
p(\hat{\I}^\*_\A) + p(\hat{\J}^\*_\A) & \leqslant & {\rm{log\,}}n\,(p(\I^\*_\A)+p(\J^\*_\A)), \label{eq:four}\\
p(\I^\*,\J^\*) &\geqslant & \dfrac{ p(\hat{\I}^\*_\A) +  p(\hat{\J}^\*_\A)}{2\,({\rm{log\,}}n\,)} + \dfrac{p(\I^\*_\C,\J^\*_\C)}{2},\nonumber \\
&\geqslant & \dfrac{p(\hat{\I}^\*_\A,\hat{\J}^\*_\A) + p(\I^\*_\C,\J^\*_\C)}{2\,({\rm{log\,}}n\,)}, \\ \label{eq:six}
& = &  \dfrac{p(\I_A,\J_A)}{2\,({\rm{log\,}}n\,)}. \nonumber
\end{eqnarray}
Equation~\eqref{eq:four} holds as $\hat{\I}^\*_\A$ and $\hat{\J}^\*_\A$ are approximate solutions to the minimum cost accessibility problem and the minimum cost sensability problem respectively, obtained using greedy
approximation of their weighted set cover formulations. Equation~\eqref{eq:six} holds as $2{\rm{log\,}}n\,) \geqslant 1$. This proves~(ii). 

From Proposition~\ref{prop:inapprox} we know that the weighted set cover problem cannot be approximated to factor $(1-o(1))\,{\rm log\,}N$, where $N$ is the cardinality of the universe. Hence, there does not exist any polynomial algorithm that has approximation ratio  $(1-o(1))\,{\rm log\, (max}(q,k))$ for Problem~\ref{prob:one}. Note that ${\rm max}(q,k) \leqslant n$. Thus  there does not exist any polynomial algorithm that has approximation ratio $(1-o(1))\,{\rm log\,}n$ for solving Problem~\ref{prob:one}. Thus the proposed algorithm is order optimal approximation algorithm for Problem~\ref{prob:one}.\qed

 In the following theorem we give the complexity of the proposed approximation algorithm.
\begin{thm}\label{th:comp}
Algorithm~\ref{alg:threestage} which takes as input a structural system $(\bA, \bB, \bC)$ with complete feedback matrix $\bK$ and cost cost vectors $p_u, p_y$ and gives as output an approximate solution $(\I_a, \J_a)$ to Problem~\ref{prob:one} has complexity $O(n^3)$, where $n$ denotes the number of states in the system.
\end{thm}
\begin{proof}
Given state digraph $\D(\bA) = \D(V_X, E_X)$ all the non-top linked SCCs can be found in $O({\rm max}(|V_X|,|E_X|))$ computations. Here $|V_X| = n$ and $|E_X|$ is atmost $|V_X|^2$. Thus set cover problems can be formulated in $O(n^2)$ computations. The greedy selection scheme for finding the approximate solution to the set cover problem has $O(n)$ complexity \cite{Chv:79}. The minimum cost bipartite matching can be solved in $O(n^3)$ computations. Thus Algorithm~\ref{alg:threestage} has $O(n^3)$ complexity. 
\end{proof}
In the next section we discuss few special class of systems in the context of Problem~\ref{prob:one}.
%%%%%%%%%%%%%%%%%%%%%%%%%%%%%%%%%%%%%%%%%%%%%%%%%%%%%%%%%%%%%
\section{Special Cases}\label{sec:cases}
In this section we consider few special cases. Using the  approximation algorithm given in Section~\ref{sec:main} we obtain the approximation results for these cases. In the following subsections we explain each of these cases briefly.

\subsection{Irreducible Systems}\label{sec:irr}
In this sub-section we consider systems whose digraph $\D(\bA)$ is irreducible, that is $\D(\bA)$ is a single SCC. Note that for this class of systems Problem~\ref{prob:one} is not NP-hard \cite{PeqKarPap:15}. Pequito.et al addressed Problem~\ref{prob:one} along with costs for feedback edges in \cite{PeqKarPap:15} and obtained a polynomial time optimal algorithm. In the following result we prove that the polynomial time algorithm given in this paper also gives an optimal solution to Problem~\ref{prob:one}.
\begin{thm}\label{th:irr}
Consider a structural system $(\bA, \bB, \bC)$, complete feedback matrix $\bK$ and cost vectors $p_u, p_y$. Let $\D(\bA)$ is irreducible. Then Algorithm~\ref{alg:threestage} returns an optimal solution to Problem~\ref{prob:one}.
\end{thm}
\begin{proof}
Given $\D(\bA)$ is irreducible and $\bK$ is complete. Thus condition~a) is satisfied by any $(y_j, u_i)$ edge. Hence Algorithm~\ref{alg:threestage} solves only the minimum cost perfect matching problem for satisfying condition~b) optimally. Without loss of generality, let $u_i$ be an input and $y_j$ be an output obtained in the solution, i.e, $i \in \I_a$ and $j\in \J_a$. Then edge $(y_j, u_i)$ satisfies both conditions in Proposition~\ref{prop:SFM1}. In case if $\B(\bA)$ has a perfect matching, then connecting the minimum cost input to the minimum cost output satisfies both the conditions in Proposition~\ref{prop:SFM1}. Thus $p(\I_a, \J_a) = p^\*$. Hence, Algorithm~\ref{alg:threestage} gives an optimal solution to Problem~\ref{prob:one}. 
\end{proof}
%%%%%%%%%%%%%%%%%%%%%%%%%%%%%%%%%%%%%%%%%%%%%%%%%%%%%%%%%%%%%%%%
\subsection{Systems with Perfect matching in $\B(\bA)$}\label{sec:match}
In this sub-section we consider systems whose bipartite graph $\B(\bA)$ has a perfect matching. In this case condition~b) in Proposition~\ref{prop:SFM1} is satisfied without using any input or output. Thus condition~a) alone has to be considered. That is, only minimum cost accessibility and minimum cost sensability problems need to be solved. We have the following result for these class of systems.
\begin{thm}\label{th:match}
Consider a structural system $(\bA, \bB, \bC)$, complete feedback matrix $\bK$ and cost vectors $p_u, p_y$. Let $\B(\bA)$ has a perfect matching. Then, Algorithm~\ref{alg:threestage} gives a $2\,({\rm log}\mu_{\rm max} + {\rm log}\eta_{\rm max})$-optimal solution to Problem~\ref{prob:one}, where $\mu_{\rm max}$ denotes the maximum number of non-top linked SCCs covered by a single input and $\eta_{\rm max}$ denotes the maximum number of non-bottom linked SCCs covered by a single output.
\end{thm}
\begin{proof}
Given $\B(\bA)$ has a perfect matching. Thus condition~b) is satisfied. Thus we need to solve only the minimum cost accessibility problem and the minimum cost sensability problem. Now following the similar lines given in the proof of Theorem~\ref{th:main}, we get $p(\I_a, \J_a) \leqslant 2\,({\rm log}\mu_{\rm max} + {\rm log}\eta_{\rm max})p^\*$. Hence, Algorithm~\ref{alg:threestage} gives a $2\,({\rm log}\mu_{\rm max} + {\rm log}\eta_{\rm max})$-optimal solution to Problem~\ref{prob:one}. 
\end{proof}
%%%%%%%%%%%%%%%%%%%%%%%%%%%%%%%%%%%%%%%%%%%%%%%%%%%%%%%%%%%%%%%%
\subsection{Systems with a Single non-top/non-bottom linked SCC}\label{sec:single}
In this sub-section we consider systems that has a single non-top linked SCC or a single non-bottom linked SCC. For this class of systems we have the following result.   
\begin{thm}\label{th:single}
Consider a structural system $(\bA, \bB, \bC)$, complete feedback matrix $\bK$ and cost vectors $p_u, p_y$. Let $\D(\bA)$ has a single non-top linked SCC. Then, Algorithm~\ref{alg:threestage} gives a $3\,({\rm log\,}\eta_{\rm max})$-optimal solution to Problem~\ref{prob:one}, where $\eta_{\rm max}$ denotes the maximum number of non-bottom linked SCCs covered by a single output.
\end{thm}
\begin{proof}
Given $\D(\bA)$ has a single non-top linked SCC. Thus $\mu_{\rm max} = 1$. Thus $p(\I_a, \J_a) \leqslant 3\,({\rm log\,}\eta_{\rm max}) p^\*$. Hence, Algorithm~\ref{alg:threestage} gives a $3\,({\rm log\,}\eta_{\rm max})$-optimal solution to Problem~\ref{prob:one}. 
\end{proof}
Note that if $\D(\bA)$ has a single non-bottom linked SCC using the same argument we will get a $3\,{\rm log\,}(\mu_{\rm max})$-optimal solution to Problem~\ref{prob:one} using Algorithm~\ref{alg:threestage}.
%%%%%%%%%%%%%%%%%%%%%%%%%%%%%%%%%%%%%%%%%%%%%%%%%%%%%%%%%%%%%
\subsection{Discrete Systems}\label{sec:discrete}
In this subsection we consider linear time invariant discrete control system given by, $x(t+1) = Ax(t) + Bu(t)$, $y(t) = Cx(t)$. For discrete systems we have the following result.
\begin{thm}\label{th:discrete}
Consider a discrete structural system $(\bA, \bB, \bC)$, complete feedback matrix $\bK$ and cost vectors $p_u, p_y$.  Then, Algorithm~\ref{alg:threestage} gives a $2\,({\rm log}\,\mu_{\rm max} + {\rm log}\,\eta_{\rm max})$-optimal solution to Problem~\ref{prob:one}.
\end{thm}
\begin{proof}
In discrete linear time invariant systems, only condition~a) in Proposition~\ref{prop:SFM1} has to be satisfied, since uncontrollable and unobservable modes of the system at origin is not of concern. Thus Algorithm~\ref{alg:threestage} need to solve only the minimum cost accessibility problem and the minimum cost sensability problem. Hence, we can get a $2\,({\rm log}\,\mu_{\rm max} + {\rm log}\,\eta_{\rm max})$-optimal solution to the minimum cost constrained input-output selection for generic arbitrary pole placement of discrete systems. 
\end{proof}
This completes the discussion of the approximation results for various special classes of systems considered. 
%%%%%%%%%%%%%%%%%%%%%%%%%%%%%%%%%%%%%%%%%%%%%%%%%%%%%%%%%%55555
\section{Conclusion}\label{sec:conclu}
This paper deals with minimum cost constrained input-output selection problem for generic arbitrary pole placement when the input and output matrices are constrained and each input and output is associated with costs. Our aim is to find a minimum cost input-output set that generic arbitrary pole placement is possible. There do not exist polynomial time algorithms for solving this unless P~=~NP. To this end, we proposed a polynomial time algorithm for finding an approximate solution to the problem by splitting the problem in to three sub-problems: minimum cost accessibility, minimum cost sensability and minimum cost disjoint cycle. We proved that minimum cost accessibility and minimum cost sensability problems are equivalent to the weighted set cover problem. Further, we proved that the minimum cost disjoint cycle problem can be solved using a minimum cost perfect matching problem on a system bipartite graph with suitably defined cost function. Using these results we proposed a polynomial time algorithm for solving minimum cost constrained input-output selection problem for generic arbitrary pole placement. The proposed algorithm gives a $3\,({\rm log}\,\mu_{\rm max} + {\rm log}\,\eta_{\rm max})$-optimal solution. We also proved that there does not exist any polynomial time algorithm that that can give a $(1-o(1))\,{\rm log\,}n$-optimal solution. Thus the proposed algorithm gives an order optimal $O({\rm log\,} n)$ approximate solution to the minimum cost input-output selection
for generic arbitrary pole placement problem.
%%%%%%%%%%%%%%%%%%%%%%%%%%%%%%%%%%%%%%%%%%%%%%%%%%%%%%%%%%%%
\bibliographystyle{myIEEEtran}  
\bibliography{CC_IO} 
%%%%%%%%%%%%%%%%%%%%%%%%%%%%%%%%%%%%%%%%%%%%%%%%%%%%%%%%%%%%%%%%
\end{document}